\documentclass[final,12pt,3p]{elsarticle}

\usepackage{graphicx}
\usepackage{amssymb}
\usepackage{amsthm}
\usepackage{amsmath}
\usepackage{algorithm}
\usepackage{algorithmic}
\usepackage{tikz}
\usepackage{lineno}
\usepackage{hyperref}
\usepackage{caption}
\usepackage{geometry}
\biboptions{sort&compress}
\hypersetup{
colorlinks=true,
linkcolor=red,
filecolor=green,
urlcolor=green,
anchorcolor=green,
citecolor=cyan,
pdftitle={Overleaf Example},
pdfpagemode=FullScreen,
}
\newtheorem{theorem}{Theorem}

\newtheorem{remark}{Remark}

\begin{document}

\begin{frontmatter}

\title{Stability Analysis of a Diffusive SVIR Epidemic Model with Distributed Delay, Imperfect Vaccine and General Incidence Rate}

\author[1]{Achraf Zinihi}
\ead{a.zinihi@edu.umi.ac.ma}

\author[1]{Mostafa Tahiri}
\ead{m.tahiri@umi.ac.ma}

\author[1]{Moulay Rchid Sidi Ammi\corref{Corr}}
\cortext[Corr]{Corresponding author}
\ead{rachidsidiammi@yahoo.fr}

\address[1]{Department of Mathematics, MAIS Laboratory, AMNEA Group, Faculty of Sciences and Technics,\\
Moulay Ismail University of Meknes, Errachidia 52000, Morocco}


\begin{abstract}
In this chapter, we consider a reaction-diffusion SVIR infection model with dis-tributed delay and nonlinear incidence rate. The wellposedness of the proposed model is proved. By means of Lyapunov functionals, we show that the disease-free equilibrium state is globally asymptotically stable when the basic reproduction number is less or equal than one, and that the disease endemic equilibrium is globally asymptotically stable when the basic reproduction number is greater than one. Numerical simulations are provided to illustrate the obtained theoretical results.
\end{abstract}


\journal{\href{https://sites.google.com/view/imame24/call-of-papers}{CRC Press Book in Mathematical Analysis}}

\date{May 16, 2024}


\end{frontmatter}


\section{Introduction}\label{S1}
The global health security remains under the persistent threat of infectious diseases \cite{Marinov2022}. In the last few decades, a plethora of mathematical models have emerged to evaluate and address the transmission dynamics of infectious diseases worldwide \cite{Annas2020, Chen2020}. These modeling frameworks play a pivotal role in diverse domains, including policy formulation, public health readiness, risk assessment, and the evaluation of control programs \cite{Salman2021, Yang2022, Zinihi2024}. Commonly, these models represent the infection status of individuals through the categorization of populations into various compartments, such as susceptible, infected, removed, and other relevant categories like vaccinated. These compartments are typically structured as ODEs, wherein the progression of infection takes place among these compartments. However, such compartmental ODE systems often neglect the impact of population movement processes and spatial heterogeneity. Recent advancements in the field have introduced reaction–diffusion models utilizing PDEs \cite{SidiAmmi2020, SidiAmmi2023}, integrating the spatial dispersal component into the ODE systems. This incorporation enables the exploration of disease spread across diverse geographical regions and the examination of spatial heterogeneity's influence on transmission patterns.

A prominent instance of the PDE model is the reaction-diffusion SVIR epidemiological system \cite{Wang2021, Zhang2019}, serving as a widely employed mathematical framework for understanding the dynamics of infectious diseases. This system encapsulates SVIR type kinetics, delineating the population's evolution through four compartments: Susceptible (S), Vaccinated (V), Infected (I), and Removed (R). The SVIR model's dynamics are subject to various influencing factors, including disease transmission rate \cite{Jamiluddin2021}, and the recovery (or treatment) rate of infected individuals \cite{Zhang2019}. Numerical simulations have become a frequent choice for studying this epidemiological model, enabling researchers to analyze disease dynamics. Stability analysis, has also found application in comprehending the qualitative characteristics of dynamical systems, encompassing both ODE models \cite{Duan2014} and PDE systems \cite{Ahmed2021}. Through the examination of stability properties, valuable insights can be derived into a dynamical system's long-term behavior, elucidating aspects such as convergence to equilibrium, manifestation of damped oscillatory behavior, and other dynamic patterns \cite{Wang2021, Zhang2019}. However, when extending stability analysis to PDE systems using numerical methods, challenges arise in terms of accuracy and complexity \cite{Mohd2019}. The traditional approach involves discretizing spatial derivatives in the PDE model and estimating eigenvalues from a considerably large discrete system. This method can be computationally demanding and susceptible to errors arising from spectral intricacies introduced during the discretization process \cite{Ahmed2021}.

The exploration of a reaction-diffusion SVIR model with distributed delay offers a compelling avenue for research with multifaceted motivations. Distributed delays in epidemic models are instrumental in capturing the intricate interplay between the incubation period of disease and the spatial dispersal of individuals. By incorporating distributed delays, the model accommodates variations in the time it takes for an infected individual to transmit the disease to others, thus reflecting more realistic scenarios in the context of infectious disease dynamics. This approach enables a nuanced examination of the spatial and temporal aspects of disease spread, providing insights into how delayed transmission influences the overall epidemiological landscape. Studying such a model is essential for gaining a deeper understanding of the impact of spatial heterogeneity, time delays, and diffusion processes on the persistence, propagation, and control of infectious diseases.\\
The rate of disease incidence, a critical factor in epidemic modeling, denotes the number of new infections occurring per unit of time. Traditionally, this rate has been assumed to exhibit bilinear dependence on the number of susceptible (S) and infected (I) individuals, and sometimes vaccinated individuals (V) and infected individuals (I). However, alternative functions have been proposed to capture disease transmission dynamics. Generally, these functions can be expressed as $f(I)$. Various specific forms of these incidence functions have been explored in the field of mathematical epidemiology, including:
\begin{itemize}
\item $\beta I$ (bilinear) \cite{Anderson1979};

\item $\beta e^{-m I} I$ with $m > 0$ \cite{Cui2007};

\item $\beta \frac{I}{1 + a_1 I}$ (saturated with respect to infectives) \cite{Kumar2018};

\item $\beta \frac{I}{1 + \omega_1 I + \omega_2 I^2}$ with $\omega_1, \omega_2 > 0$ \cite{Zhou2007};
\end{itemize}
where $\beta$ transmission rate for the susceptible or the vaccinated individuals.

The subsequent sections of this manuscript are structured as follows. In Section \ref{S2}, the SVIR model incorporating distributed delay is presented. Section \ref{S3} is dedicated to establishing the existence and uniqueness of a global solution for the proposed model. Section \ref{S4} encompasses the computation of the basic reproduction number, examining model persistence, and investigating the existence of endemic equilibria. The focus then shifts to the Global stability analysis of both the disease-free equilibrium and the endemic equilibrium in Section \ref{S5}. To complement the theoretical findings, Section \ref{S6} presents a series of numerical simulations. A closure for this work will be presented in Section \ref{S7}.


\section{Mathematical model}\label{S2}

Depending on what we discussed earlier, our proposed model is defined by
\begin{equation}\label{E2.1}
\begin{cases}
\begin{aligned}
\frac{\partial S(t, x)}{\partial t}-d_S \Delta S(t, x) =& \Lambda-S(t, x) \int_0^k g(\tau) f(I(t-\tau, x)) d \tau-(\mu+\alpha) S(t, x),\\
\frac{\partial V(t, x)}{\partial t}-d_V \Delta V(t, x) =& \alpha S(t, x)-V(t, x) \int_0^k g(\tau) h(I(t-\tau, x)) d \tau-\left(\gamma_1+\mu\right) V(t, x),\\
\frac{\partial I(t, x)}{\partial t}-d_I \Delta I(t, x) =& V(t, x) \int_0^k g(\tau) h(I(t-\tau, x)) d \tau\\
&+S(t, x) \int_0^k g(\tau) f(I(t-\tau, x)) d \tau-(\gamma+\mu+c) I(t, x),\\
\frac{\partial R(t, x)}{\partial t}-d_R \Delta R(t, x) =& \gamma_1 V(t, x)+\gamma I(t, x)-\mu R(t, x),\\
\end{aligned}
\end{cases}
\end{equation}
with
\begin{equation}\label{E2.2}
\frac{\partial S(t, x)}{\partial \nu}=\frac{\partial V(t, x)}{\partial \nu}=\frac{\partial I(t, x)}{\partial \nu}=\frac{\partial R(t, x)}{\partial \nu}=0, \quad x \in \partial \Omega,
\end{equation}
where $t>0$, $\Omega$ is a bounded domain in $\mathbb{R}^n$ with smooth boundary $\partial \Omega$, $\nu$ is the outward normal to $\partial \Omega,$ $  d_S>0,$ $d_V>0,$ $d_I>0$ and $d_R>0$ stand for the diffusion rates, $S(t,x), V(t,x), I(t,x), R(t,x)$ denote the number of susceptible, vaccinated, infected and recovered individuals at time $t$ and position $x$, respectively. $\Lambda$ is the recruitment rate of new susceptibles, $\mu$ is the natural death rate of the population, $\gamma$ represents the natural recovery rate of infective individuals, $\alpha$ is the vaccination rate. $\frac{1}{\gamma_1}$ is the approximative time spend in V-class before getting immunity. $\frac{1}{\gamma}$ the average time of the infectious period. Individuals leave S-class with the following rate
$$
S(t, x) \int_0^k g(\tau) f(I(t-\tau, x)) d \tau,
$$
and leave V-classe with the rate
$$
V(t, x) \int_0^k g(\tau) h(I(t-\tau, x)) d \tau,
$$
where $k$ represents the maximum time taken to become infectious and $g$ is a non-negative function satisfying 
$$\int_0^k g(\tau) d \tau=1.$$
The initial condition for the above system \eqref{E2.1} is give for $\theta \in[-k, 0]$ by
\begin{equation}\label{E2.3}
\begin{aligned}
\Phi(\theta)(x) & =\left(\Phi_1(\theta, x), \Phi_2(\theta, x), \Phi_3(\theta, x), \Phi_4(\theta, x)\right) \\
& =(S(\theta, x), V(\theta, x), I(\theta, x), R(\theta, x)), x \in \bar{\Omega}.
\end{aligned}
\end{equation}

The function $\Phi$ belongs $C([-k, 0], \mathbb{X})$, where $C([-k, 0], \mathbb{X})$ denotes the space of continuous functions mapping from $[-k, 0]$ to $\mathbb{X}$ equipped with the sup-norm and $\mathbb{X}=C\left(\Omega, \mathbb{R}_+^4\right)$ denotes the space of continuious functions mapping from $\Omega$ to $\mathbb{R}_+^4$.\\
Our main objective is to discuss the Global stability of the SVIR model \eqref{E2.1}. For that, we will construct suitable Lyapunov functions. Throughout this
work we assume that $f: \mathbb{R}_+ \longrightarrow \mathbb{R}_+$ and $h: \mathbb{R}_+ \longrightarrow \mathbb{R}_+$ are continuously differentiable in the interior of $\mathbb{R}_+$ with
$f(I)=0$ for $I = 0$ and $h(I)=0$ for $I = 0$
and the following hypotheses hold
$$
\begin{aligned}
& (H_1) \text{ For } I>0, \text{ we have } f(I)>0 \text{ and } h(I)>0.\\
& (H_2) \text{ If } I \geq 0, \text{ then } f^{\prime}(I)>0, \  h^{\prime}(I)>0, \ f^{\prime \prime}(I) \leq 0 \text{ and } h^{\prime \prime}(I) \leq 0.
\end{aligned}
$$
Epidemiologically, $\left(H_1\right)$ means that individuals are positive. $\left(H_2\right)$ implies that the incidences of
$$
V(t, x) \int_0^k g(\tau) h(I(t-\tau, x)) d \tau,
$$
and
$$
S(t, x) \int_0^k g(\tau) f(I(t-\tau, x)) d \tau,
$$
become faster with an increase in the number of the infectious individuals. However, the per capita infection rate will slow downbecause of a certain inhibition effect, since $f^{\prime \prime}(I) \leq 0$ and $h^{\prime \prime}(I) \leq 0$ imply that $\left(\frac{f(I)}{I}\right)^{\prime}$, $\left(\frac{h(I)}{I}\right)^{\prime}<0$.


\section{Basic properties of the model}\label{S3}

Let $A$ be the operator defined on $\mathbb{X}$ as follows
\begin{equation}\label{E3.1}
\begin{aligned}
A: D(A) \subset \mathbb{X} & \longrightarrow \mathbb{X}, \\
u & \longmapsto Au = \left(A_1 u, A_2 u, A_3 u, A_4 u\right),
\end{aligned}
\end{equation}
where
\begin{equation}\label{E3.2}
D(A):=\left\{u \in \mathbb{X}: \Delta u \in \mathbb{X}, \frac{\partial u}{\partial \nu}=0 \text { on } \partial \Omega\right\},
\end{equation}
and
$$
\begin{aligned}
& A_1 u = (d_S \Delta - \delta_1) u, \ \delta_1 = \mu + \alpha,\\
& A_2 u = (d_V \Delta - \delta_2)  u, \  \delta_2 = \mu + \gamma_1,\\
& A_3 u = (d_I \Delta - \delta_3)  u, \ \delta_3 = \mu + \gamma + c,\\
& A_4 u = (d_R \Delta - \delta_4)  u, \ \delta_4 = \mu.
\end{aligned}
$$
According to the classical theory of semi-groups of partial differential equations \cite{Pazy1983}, $A_i$ is the infinitesimal generator of a strongly continuous semi-group $\exp (t A_i)$, where $i=1,2,3,4$.\\

Let $\mathcal{X}_i(t): C\left(\Omega, \mathbb{R}\right) \longrightarrow C\left(\Omega, \mathbb{R}\right) $ be the $C_0$ semi-group associated with $A_i$ subject to \eqref{E2.2}. Then,
$$
\left(\mathcal{X}_i(t)u\right)(x) = \int_\Omega \Gamma_i(t,x,y) u(y) dy,
$$
where $\Gamma_i$ represents the Green function associated with $A_i$ subject to \eqref{E2.2}. By \cite[Corollary 7.2.3]{Smith1995}, we obtain that $\mathcal{X}(t) = \left(\mathcal{X}_1(t), \mathcal{X}_2(t), \mathcal{X}_3(t), \mathcal{X}_4(t)\right)$ is compact and strongly positive for each $t > 0$.\\
Following \cite{Luo2019, Ren2018}, there exist constants $\xi_i > 0$ such that
\begin{equation}\label{E3.3}
\|\mathcal{X}_i(t)\| \leq \xi_i e^{\lambda_i t}, \ \forall t \geq 0,
\end{equation}
where $\lambda_i$ is the principal eigenvalue of $A_i$ subject to \eqref{E2.2}.

For any function $u:[-k, \sigma) \longrightarrow \mathbb{X}$ with some $\sigma>0$, we define $u_t \in C([-k, 0], \mathbb{X})$ by $u_t(\theta)=u(t+\theta), \theta \in[-k, 0]$. Let $F$ be a function defined by
$$
\begin{aligned}
F: C([-k, 0], \mathbb{X}) & \longrightarrow C\left(\Omega, \mathbb{R}^{4}\right),\\
\phi & \longmapsto F(\phi) = \left(F_1(\phi), F_2(\phi), F_3(\phi), F_4(\phi)\right),
\end{aligned}
$$
where
$$
\begin{aligned}
F_1(\phi) =& \Lambda-\phi_1(t, x) \int_0^k g(\tau) f\left(\phi_3(t-\tau, x)\right) d \tau,\\
F_2(\phi) =& \alpha \phi_1(t, x)-\phi_2(t, x) \int_0^k g(\tau) h\left(\phi_3(t-\tau, x)\right) d \tau,\\
F_3(\phi) =& \phi_2(t, x) \int_0^k g(\tau) h\left(\phi_3(t-\tau, x)\right) d \tau+\phi_1(t, x) \int_0^k g(\tau) f\left(\phi_3(t-\tau, x)\right) d \tau,\\
F_4(\phi) =& \gamma_1 \phi_2(t, x)+\gamma \phi_3(t, x).
\end{aligned}
$$
Then, system \eqref{E2.1}-\eqref{E2.3} can be written in the following abstract form
\begin{equation}\label{E3.4}
\left\{\begin{aligned}
&\frac{d \vartheta(t)}{d t}=A \vartheta(t)+F\left(\vartheta_t\right), \ t \geq 0,\\
&\vartheta_0=\Phi,
\end{aligned}\right.
\end{equation}
where $\vartheta(t) = \left(S(t), V(t), I(t), R(t)\right)^T$ and $\Phi=(S_0, V_0, I_0, R_0)^T$. In addition, the proposed system \eqref{E2.1} can be rewritten as the integral equation
\begin{equation}\label{E3.5}
\vartheta(t) = \mathcal{X}(t)\Phi + \int_0^t \mathcal{X}(t-s) F(\vartheta_s) ds.
\end{equation}
Function $F$ is locally Lipschitz on $C([-k, 0], \mathbb{X})$, then problem \eqref{E2.1} has a unique local non-negative solution $\left(S(t), V(t), I(t), R(t)\right)^T$ for all $t \in [0, b)$.\\

\begin{remark}
The unique local positive solution becomes the global one if $b=+\infty$.\\
\end{remark}

We will now show that the local solution can be extended to a global one. Suppose that $b < +\infty$. By \cite[Theorem 2]{Martin1990}, we know that
\begin{equation}\label{E3.6}
\|\vartheta(t)\| \longrightarrow +\infty \text{ as } t \longrightarrow b. 
\end{equation}
Put 
$$
N(t) = \int_{\Omega}\{S(t, x)+V(t, x)+I(t, x)+R(t, x)\} dx, \quad \forall t \in [0, b).
$$
By the divergence theorem \cite[Theorem 3.7]{Groeger2014} and \eqref{E2.2}, we get
$$
\begin{aligned}
& \int_{\Omega} d_S \Delta S(t, x) dx = 0, \ \int_{\Omega} d_V \Delta V(t, x) dx = 0, \\
& \int_{\Omega} d_I \Delta I(t, x) dx = 0, \ \int_{\Omega} d_R \Delta R(t, x) dx =0 .
\end{aligned}
$$
Which gives
$$
\frac{d N(t)}{d t} \leq \Lambda|\Omega|-\mu N(t).
$$
By the comparison principle, there exists constants $M_1 > 0$ and $t_1 > 0$ such that $N(t) \leq M_1$ for all $t \in [t_1, b)$. Consequently,
\begin{equation}\label{E3.7}
\begin{aligned}
& \int_{\Omega} S(t, x) dx \leq M_1, \quad \int_{\Omega} V(t, x) dx \leq M_1,\\
& \int_{\Omega} I(t, x) dx \leq M_1, \quad \int_{\Omega} R(t, x) dx \leq M_1.
\end{aligned} \quad \forall t \in [t_1, b),
\end{equation}
Let $\lambda_i^j$ the eigenvalue of $A_i$ subject to \eqref{E2.2} corresponding to the eigenfunction $\varphi_i^j$, such that $0\geq -\delta_i = \lambda_i^1 > \lambda_i^2 > \lambda_i^3 > \cdots$. From \cite[Chapter 5]{Guenther1996}, we have
$$
\Gamma_i(t,x,y) = \sum_{j\geq 1} \exp(\lambda_i^j t) \varphi_i^j(x) \varphi_i^j(y).
$$
Since $\varphi_i^j$ is uniformly bounded, then there exists $\beta > 0$ such that
$$
\Gamma_i(t,x,y) \leq \beta \sum_{j\geq 1} \exp(\lambda_i^j t), \quad t > 0, \ x,y\in\Omega.
$$
Adopting a similar methodology as outlined in \cite[Page 5]{AvilaVales2021} and using \eqref{E3.3} together with \eqref{E3.5} and \eqref{E3.7}, we obtain
$$
\begin{aligned}
& S(t, x) < \infty, \quad V(t, x) < \infty,\\
& I(t, x) < \infty, \quad R(t, x) < \infty.
\end{aligned} \quad \forall (t,x) \in [0, b)\times\Omega.
$$
Which contradicts \eqref{E3.6}. Thus, $b = +\infty$.
These analyses lead to the following inevitable result.

\begin{theorem}
For any initial value function in $\mathbb{X}$, system \eqref{E2.1} has a unique non-negative solution $\left(S(t, x), V(t, x), I(t, x), R(t, x)\right)$ for all $(t,x) \in [0, \infty)\times\Omega$.
\end{theorem}


\section{Equilibria and the basic reproduction number}\label{S4}

Since the first three equations on \eqref{E2.1} do not contain $R(t, x)$, it is sufficient to analyze the behavior of solutions to the following system
\begin{equation}\label{E4.1}
\left\{\begin{aligned}
\frac{\partial S(t, x)}{\partial t}-d_S \Delta S(t, x) =& \Lambda-S(t, x) \int_0^k g(\tau) f(I(t-\tau, x)) d \tau-(\mu+\alpha) S(t, x), \\
\frac{\partial V(t, x)}{\partial t}-d_V \Delta V(t, x) =& \alpha S(t, x)-V(t, x) \int_0^k g(\tau) h(I(t-\tau, x)) d \tau-\left(\gamma_1+\mu\right) V(t, x),\\
\frac{\partial I(t, x)}{\partial t}-d_I \Delta I(t, x) =& V(t, x) \int_0^k g(\tau) h(I(t-\tau, x)) d \tau\\
&+S(t, x) \int_0^k g(\tau) f(I(t-\tau, x)) d \tau-(\gamma+\mu+c) I(t, x),
\end{aligned}\right. \quad x \in \Omega,
\end{equation}
with
\begin{equation}\label{E4.2}
\frac{\partial S(t, x)}{\partial \nu}=\frac{\partial V(t, x)}{\partial \nu}=\frac{\partial I(t, x)}{\partial \nu}=0, \quad x \in \partial \Omega.
\end{equation}

System \eqref{E4.1} always has a disease-free equilibrium $E_0=\left(S_0, V_0, 0\right)^T$, where $S_0=\frac{\Lambda}{\mu+\alpha}$ and $V_0=\frac{\alpha \Lambda}{(\mu+\alpha)\left(\gamma_1+\mu\right)}$. Furthermore, by a simple and direct calculation, we conclude that the basic reproduction number for this proposed model \eqref{E4.1} is given by
$$
\mathcal{R}_0=\frac{S_0}{(\gamma+\mu+c)} f^{\prime}(0)+\frac{V_0}{(\gamma+\mu+c)} h^{\prime}(0).
$$

We have the following result.

\begin{theorem}
If $\mathcal{R}_0>1$, then system \eqref{E4.1} has a unique endemic equilibrium $E^*=$ $\left(S^*, V^*, I^*\right)$ with $S^*=\frac{\Lambda}{f\left(I^*\right)+\mu+\alpha}$ and $V^*=\frac{\alpha \Lambda}{\left(h\left(I^*\right)+\gamma_1+\mu\right)\left(f\left(I^*\right)+\mu+\alpha\right)}$.
\end{theorem}

\begin{proof}
When $\mathcal{R}_0>1$, the result is obvious. According to the first two equations of \eqref{E4.1} one can get
$$
S=\frac{\Lambda}{f(I)+\mu+\alpha}, \quad V=\frac{\alpha \Lambda}{\left(h(I)+\gamma_1+\mu\right)(f(I)+\mu+\alpha)} .
$$
Then, we have
$$
H(I)=\frac{\Lambda}{f(I)+\mu+\alpha} f(I)+\frac{\alpha \Lambda}{\left(h(I)+\gamma_1+\mu\right)(f(I)+\mu+\alpha)} h(I)-(\gamma+\mu+c) I .
$$
Obviously, $H(+\infty)=-\infty$ and $H(0)=0$. It follows from $H^{\prime}(0)>0$ that $H(I)=0$ has at least one positive solution denoted by $I^*$, where
$$
\begin{aligned}
H^{\prime}(0) &=\frac{\Lambda}{\mu+\alpha}+\frac{\alpha \Lambda h^{\prime}(0)}{\left(\gamma_1+\mu\right)(\mu+\alpha)}-(\gamma+\mu+c)\\
&=(\gamma+\mu+c)\left(\mathcal{R}_0-1\right)>0.
\end{aligned}
$$
This is equivalent to $\mathcal{R}_0>1$. Thus, \eqref{E4.1} has at least one positive solution with
$$
S^*=\frac{\Lambda}{f\left(I^*\right)+\mu+\alpha}, \quad V^*=\frac{\alpha \Lambda}{\left(h\left(I^*\right)+\gamma_1+\mu\right)\left(f\left(I^*\right)+\mu+\alpha\right)} .
$$
Now, we prove that the endemic equilibrium is unique. Note that
$$
\begin{aligned}
H^{\prime}(I)=& \ \frac{\Lambda f^{\prime}(I)(f(I)+\mu+\alpha)-\Lambda f(I) f^{\prime}(I)}{(f(I)+\mu+\alpha)^2} + \frac{\alpha \Lambda h^{\prime}(I)}{\left[\left(h(I)+\gamma_1+\mu\right)(f(I)+\mu+\alpha)\right]} \\
& -\frac{\alpha \Lambda h(I)\left[h^{\prime}(I)(f(I)+\mu+\alpha)+\left(h(I)+\gamma_1+\mu\right) f^{\prime}(I)\right]}{\left[\left(h(I)+\gamma_1+\mu\right)(f(I)+\mu+\alpha)\right]^2} -(\gamma+\mu+c)
\end{aligned}
$$
and
$$
\begin{aligned}
H^{\prime \prime}(I)=& \frac{\Lambda f^{\prime \prime}(I) \mu}{(f(I)+\mu+\alpha)^3}+\frac{\alpha \Lambda f^{\prime \prime}(I)\left(\gamma_1 \mu\right)}{(f(I)+\mu+\alpha)^2\left(h(I)+\gamma_1+\mu\right)}\\
&+\frac{\alpha \Lambda h^{\prime \prime}(I)\left(\gamma_1+\mu\right)}{\left.(f(I)+\mu+\alpha)\left(h(I)+\gamma_1\right)+\mu\right)^2} \\
&-\left[\frac{\alpha \Lambda h^{\prime}(I) f^{\prime}(I)\left(\gamma_1+\mu\right)}{\left.(f(I)+\mu+\alpha)\left(h(I)+\gamma_1\right)+\mu\right)^2}+\frac{2 \alpha \Lambda\left(h^{\prime}(I)\right)^2\left(\gamma_1+\mu\right)}{(f(I)+\mu+\alpha)\left(h(I)+\gamma_1+\mu\right)^3}\right. \\
&+\frac{2 \Lambda \mu\left(f^{\prime}(I)\right)^2}{(f(I)+\mu+\alpha)^3}+\frac{\alpha \Lambda h^{\prime}(I) f^{\prime}(I)}{(f(I)+\mu+\alpha)^2\left(h(I)+\gamma_1+\mu\right)} \\
&\left.+\frac{2 \alpha \Lambda\left(f^{\prime}(I)\right)^2\left(\gamma_1+\mu\right)}{(f(I)+\mu+\alpha)^3\left(h(I)+\gamma_1+\mu\right)}\right]
\end{aligned}
$$
By $\left(H_2\right)$, we know that $h^{\prime \prime}(I)<0$ for $I>0$. If there exists more than one positive equilibrium, then there is a point $E^*=\left(S^*, V^*, I^*\right)$ such that $h\left(I^*\right)=0$. We obtain a contradiction.
\end{proof}


\section{Global stability of equilibria}\label{S5}

In this section, we show the global asymptotic stability of the disease-free equilibrium $E_0$ of the system \eqref{E4.1} by constructing a Lyapunov functional. The following result holds.

\begin{theorem}
Under hypotheses $\left(H_1\right)$ and $\left(H_2\right)$ the disease-free equilibrium $E_0$ of system \eqref{E4.1} is globally asymptotically stable if, and only if, $\mathcal{R}_0<1$.
\end{theorem}
 
\begin{proof}
To prove our result, we consider the following Lyapunov functional
$$
L(t)=\int_{\Omega} (L_1(t, x)+L_2(t, x)) d x
$$
where
$$
L_1(t, x)=I(t, x)+S_0 \phi\left(\frac{S(t, x)}{S_0}\right)+V_0 \phi\left(\frac{V(t, x)}{V_0}\right)
$$
with $\phi(x)=x-1-\ln (x)$ and
$$
L_2(t, x)=\int_0^k g(\tau) \int_{t-\tau}^\tau f(I(u, x)) S(u, x) d u+\int_0^k g(\tau) \int_{t-\tau}^\tau h(I(u, x)) V(u, x) d u d \tau.
$$
According to
$$
\ln x \leq x-1, \ S_0=\frac{\Lambda}{\alpha+\mu} \ \text{ and } \ \alpha S_0=\left(\gamma_1+\mu\right),
$$
we have
$$
\begin{aligned}
\frac{\partial L_1(t, x)}{\partial t}=&\frac{\partial I}{\partial t}+\left(1-\frac{S_0}{S}\right) \frac{\partial S}{\partial t}+\left(1-\frac{V_0}{V}\right) \frac{\partial V}{\partial t},\\
\frac{\partial L_2(t, x)}{\partial t}=&\int_0^k g(\tau)(f(I) S-f(I(t-\tau, x) S(t-\tau, x)) d \tau \\
&+\int_0^k g(\tau)(h(I) V-h(I(t-\tau, x) V(t-\tau, x)) d \tau.
\end{aligned}
$$
Thus, we get
$$
\begin{aligned}
\frac{\partial\left(L_1+L_2\right)(t, x)}{\partial t}=& d_I \Delta I+S \int_0^k g(\tau) f(I(t-\tau, x)) d \tau \\
&+V \int_0^k g(\tau) h(I(t-\tau, x)) d \tau-(\gamma+\mu+c) I \\
& +\left(1-\frac{S_0}{S}\right)\left(d_S \Delta S+\Lambda\right.\\
&-\left.S \int_0^k g(\tau) f(I(t-\tau, x)) d \tau-(\mu+\alpha) S\right)+ \\
& \left(1-\frac{V_0}{V}\right)\left(d_V \Delta V+\alpha S\right.\\
&-\left.V \int_0^k g(\tau) h(I(t-\tau, x)) d \tau-\left(\gamma_1+\mu\right) V\right) d \tau \\
& +\int_0^k g(\tau) f(I) S d \tau-\int_0^k g(\tau) f(I(t-\tau, x)) S(t-\tau, x) d \tau\\
&+\int_0^k g(\tau) h(I) V d \tau-\int_0^k g(\tau) h(I(t-\tau, x)) V(t-\tau, x) d \tau.
\end{aligned}
$$
That is to say
$$
\begin{aligned}
\frac{\partial\left(L_1+L_2\right)(t, x)}{\partial t}=&-(\mu+\alpha)\left(\frac{S}{S_0}+\frac{S_0}{S}-2\right)+\alpha S_0\left(\frac{S}{S_0}-\frac{V}{V_0}-\frac{V_0 S}{V S_0}\right)\\
&+\int_0^k g(\tau)\left(S_0 f(I(t-\tau, x))+f(I) S\right.\\
&\left.-f(I(t-\tau, x)) S(t-\tau, x)\right) d \tau +\int_0^k g(\tau)\left(V_0 h(I(t-\tau, x))\right.\\
&\left.+h(I) V-h(I(t-\tau, x)) V(t-\tau, x)\right) d \tau \\
&-(\gamma+\mu+c) I \\
\leq& \alpha S_0\left(\phi\left(\frac{S}{S_0}\right)-\phi\left(\frac{V}{V_0}\right)-\phi\left(\frac{V_0 S}{S_0 V}\right)-2 \ln \left(\frac{V S_0}{V_0 S}\right)\right) \\
&-S_0(\mu+\alpha)\left(\phi\left(\frac{S}{S_0}\right)+\phi\left(\frac{S_0}{S}\right)\right) + S_0 f^{\prime}(0) I+V_0 h^{\prime}(0) I\\
&-(\gamma+\mu+c) I\\
=& -\alpha S_0\left(\phi\left(\frac{S_0}{S}\right)+\phi\left(\frac{V}{V_0}\right)+\phi\left(\frac{V_0 S}{S_0 V}\right)+2 \ln \left(\frac{V S_0}{V_0 S}\right)\right) \\
&-\mu S_0\left(\phi\left(\frac{S}{S_0}\right)+\phi\left(\frac{S_0}{S}\right)\right)+(\gamma+\mu+c)\left(\mathcal{R}_0-1\right).
\end{aligned}
$$
From which we conclude that
$$
\begin{aligned}
\frac{d L(t)}{d t} \leq & \int_{\Omega}\left\{-\alpha S_0\left(\phi\left(\frac{S_0}{S}\right)+\phi\left(\frac{V}{V_0}\right)+\phi\left(\frac{V_0 S}{S_0 V}\right)+2 \ln \left(\frac{V S_0}{V_0 S}\right)\right)\right. \\
& -\mu S_0\left(\phi\left(\frac{S}{S_0}\right)+\phi\left(\frac{S_0}{S}\right)\right)+(\gamma+\mu+c)\left(\mathcal{R}_0-1\right) +d_I \Delta I\\
&+  \left.d_S \Delta S+d_V \Delta V-\frac{S_0}{S} d_S \Delta S-\frac{V_0}{V} d_V \Delta V\right\} d x.
\end{aligned}
$$
By Green's formula and \eqref{E2.2}, we obtain that
$$
\begin{gathered}
\int_{\Omega}\left(d_I \Delta I(t, x)+d_S \Delta S(t, x)+d_V \Delta V(t, x)\right) d x=0, \\
\int_{\Omega} \frac{\Delta S(t, x)}{S(t, x)} d x= \int_{\Omega} \frac{(\nabla S(t, x))^2}{S(t, x)^2} d x \geq 0,
\end{gathered}
$$
and
$$
\int_{\Omega} \frac{\Delta V(t, x)}{V(t, x)} d x= \int_{\Omega} \frac{(\nabla V(t, x))^2}{V(t, x)^2} d x \geq 0 .
$$
If $\mathcal{R}_0<1$, then
$$
\begin{aligned}
\frac{d L(t)}{d t} \leq& \int_{\Omega}\left\{-\alpha S_0\left(\phi\left(\frac{S_0}{S}\right)+\phi\left(\frac{V}{V_0}\right)+\phi\left(\frac{V_0 S}{S_0 V}\right)+2 \ln \left(\frac{V S_0}{V_0 S}\right)\right)\right. \\
&- \left.\mu S_0\left(\phi\left(\frac{S}{S_0}\right)+\phi\left(\frac{S_0}{S}\right)\right)+(\gamma+\mu+c)\left(\mathcal{R}_0-1\right)\right\} d x \leq 0.
\end{aligned}
$$
Thus, the disease-free equilibrium of \eqref{E4.1} is globally asymptotically stable.
\end{proof}

\begin{theorem}
If $\mathcal{R}_0 \geq 1$, then $E^*$ of system \eqref{E4.1} is globally asymptotically stable.
\end{theorem}

\begin{proof}
Define
$$
H(t)=\int_{\Omega} (H_1(t, x)+H_2(t, x)) d x
$$
where
$$
\begin{aligned}
H_1(t, x)=&S^* \phi\left(\frac{S}{S *}\right)+V^* \phi\left(\frac{V}{V *}\right)+I^* \phi\left(\frac{I}{I *}\right), \\
H_2(t, x)=&S^* f\left(I^*\right) \int_{t-\tau}^t g(\tau) \int_0^\tau \phi\left(\frac{S(\theta) f\left(I(\theta)\right)}{S^* f\left(I^*)\right.}\right) d \theta d \tau \\
&+V^* h\left(I^*\right) \int_{t-\tau}^t g(\tau) \int_0^\tau \phi\left(\frac{V(\theta) h\left(I(\theta)\right)}{V^* h\left(I^*)\right.}\right) d \theta d \tau .
\end{aligned}
$$
Then,
$$
\frac{\partial H_1(t, x)}{\partial t}=\left(1-\frac{S^*}{S}\right) \frac{\partial S}{\partial t}+\left(1-\frac{V^*}{V}\right) \frac{\partial V}{\partial t}+\left(1-\frac{I^*}{I}\right) \frac{\partial I}{\partial t},
$$
and
$$
\begin{aligned}
\frac{\partial H_2(t, x)}{\partial t}=&\int_0^k g(\tau)\left(S f(I)-S(\tau) f\left(I(\tau)\right)+S^* f\left(I^*\right) \ln \left(\frac{S(\tau) f\left(I(\tau))\right.}{S f(I)}\right)\right) d\tau \\
&+\int_0^k g(\tau)\left(V h(I)-V(\tau) h\left(I(\tau)\right)+V^* h\left(I^*\right) \ln \left(\frac{V(\tau) h\left(I(\tau))\right.}{V h(I)}\right)\right) d\tau .
\end{aligned}
$$
Then,
$$
\begin{aligned}
\frac{\partial \left(H_1+H_2\right)(t, x)}{\partial t}=&-S^*(\mu+\alpha)\left(\frac{S}{S^*}+\frac{S^*}{S}-2\right)\\
&+\left(1-\frac{S^*}{S}\right)\left(S^* f\left(I^*\right)-\int_0^k g(\tau) S f\left(I(\tau)\right) d \tau\right)\\
&+\left(1-\frac{S^*}{S}\right) d_S \Delta S +\left(\gamma_1+\mu\right) V^*\left(\frac{S}{S^*}-\frac{S V^*}{V S^*}-\frac{V}{V^*}+1\right)\\
&+\left(1-\frac{V^*}{V}\right)\left(\frac{S V^* h\left(I^*\right)}{\left(\gamma_1+\mu\right) S^*}-\int_0^k g(\tau) V h\left(I(\tau)\right) d \tau\right)\\
&+\left(1-\frac{V^*}{V}\right) d_V \Delta V+S^* f\left(I^*\right)+V^* h\left(I^*\right)-\frac{S^* I f\left(I^*\right)}{I^*}\\
&-\frac{I V^* h\left(I^*\right)}{I^*} +\left(1-\frac{I^*}{I}\right)\left(\int_0^k g(\tau) S f\left(I(\tau)\right)+\int_0^k g(\tau) V h\left(I(\tau)\right)\right).
\end{aligned}
$$
Thus,
$$
\begin{aligned}
\frac{\partial \left(H_1+H_2\right)(t, x)}{\partial t}=&-S^*(\mu+\alpha) \phi\left(\frac{S^*}{S}\right)+\left(\gamma_1+\mu\right) V^*\left(-\phi\left(\frac{S V^*}{V S^*}\right)-\phi\left(\frac{V}{V^*}\right)\right) \\
&-S^* f\left(I^*\right) \int_0^k g(\tau)\left[\phi \left(\frac{S^*}{S}+\phi\left(\frac{I^* S f\left(I^*\right)}{I S^* f(I)}\right)+\phi\left(\frac{S(\tau) f\left(I(\tau)\right)}{S^* f\left(I^*\right)}\right)+1\right.\right. \\
&\left.-\frac{f\left(I(\tau)\right)}{f\left(I^*\right)}-\frac{S f(I)}{S^* f\left(I^*\right)}-\ln \left(\frac{I S(\tau) \left(f(I(\tau))\right)^3}{I S f(I) (f(I))^2}\right)\right] d\tau\\
& -V^* h\left(I^*\right) \int_0^k g(\tau)\left[\phi\left(\frac{S V^*}{\left(\gamma_1+\mu\right) S^* V}\right)+\phi\left(\frac{I}{I^*}\right)+\phi\left(\frac{I^* V h\left(I(\tau)\right)}{I V^* h\left(I^*\right)}\right)\right. \\
& \left.+\phi\left(\frac{V(\tau) h\left(\left(I(\tau)\right)\right.}{V^* h\left(I^*\right)}\right)-\frac{S V^*}{V S^*}-\frac{V}{V^*}-\frac{S}{\left(\gamma_1+\mu\right) S^*}-\frac{h\left(I(\tau)\right)}{h\left(I^*\right)}-1-\frac{V h(I)}{V^* h\left(I^*\right)}\right. \\
& \left.-\ln \left(\frac{(V(\tau))^2\left(h\left(I(\tau)\right)\right)^2}{\left(\gamma_1+\mu\right) V^* h(I)}\right)\right]d\tau+d_S \Delta S-\frac{S^*}{S} d_S \Delta S\\
& +d_V \Delta V-\frac{V^*}{V} d_V \Delta V+d_I \Delta I- \frac{I^*}{I} d_I \Delta I .
\end{aligned}
$$
Note that
$$
\int_{\Omega} d_S \Delta S(t, x)=\int_{\Omega} d_V \Delta V(t, x)=\int_{\Omega} d_I \Delta I(t, x)=0,
$$
and 
$$
\int_{\Omega} \frac{\Delta S}{S} d x=\int_{\Omega} \frac{\nabla S^2}{S^2} d x \geq 0, \ \int_{\Omega} \frac{\Delta V}{V} d x=\int_{\Omega} \frac{\nabla V^2}{V^2} d x \geq 0, \ \int_{\Omega} \frac{\Delta I}{I} d x=\int_{\Omega} \frac{\nabla I^2}{I^2} d x \geq 0,
$$
then,
$$
\begin{aligned}
\frac{d H(t)}{d t}=&\int_{\Omega} \frac{\partial \left(H_1(t, x)+H_2(t, x)\right)}{\partial t} d x\\
\leq& \int_{\Omega}-S^*(\mu+\alpha) \phi\left(\frac{S^*}{S}\right)+\left(\gamma_1+\mu\right) V^*\left(-\phi\left(\frac{S V^*}{V S^*}\right)-\phi\left(\frac{V}{V^*}\right)\right) \\
&-S^* f\left(I^*\right) \int_0^k g(\tau)\left[\phi \left(\frac{S^*}{S}+\phi\left(\frac{I^* S f\left(I^*\right)}{I S^* f(I)}\right)+\phi\left(\frac{S(\tau) f\left(I(\tau)\right)}{S^* f\left(I^*\right)}\right)+1-\frac{f\left(I(\tau)\right)}{f\left(I^*\right)}\right.\right. \\
&\left.-\frac{S f(I)}{S^* f\left(I^*\right)}-\ln \left(\frac{I S(\tau) \left(f(I(\tau))\right)^3}{I S f(I) (f(I))^2}\right)\right] d\tau -V^* h\left(I^*\right) \int_0^k g(\tau)\left[\phi\left(\frac{S V^*}{\left(\gamma_1+\mu\right) S^* V}\right)\right.\\
&\left.+\phi\left(\frac{I}{I^*}\right)+\phi\left(\frac{I^* V h\left(I(\tau)\right)}{I V^* h\left(I^*\right)}\right)+\phi\left(\frac{V(\tau) h\left(\left(I(\tau)\right)\right.}{V^* h\left(I^*\right)}\right)\right. \left.-\frac{S V^*}{V S^*}-\frac{V}{V^*}-\frac{S}{\left(\gamma_1+\mu\right) S^*}\right.\\
&\left.-\frac{h\left(I(\tau)\right)}{h\left(I^*\right)}-1-\frac{V h(I)}{V^* h\left(I^*\right)}-\ln \left(\frac{(V(\tau))^2\left(h\left(I(\tau)\right)\right)^2}{\left(\gamma_1+\mu\right) V^* h(I)}\right)\right] d\tau.
\end{aligned}
$$
By assumption $\left(H_2\right)$ and $\ln (x) \leq x-1$, we can get
$$
\frac{f\left(I(\tau)\right)}{f\left(I^*\right)}+\frac{S f(I)}{S^* f\left(I^*\right)}-\ln \left(\frac{I S(\tau) \left(f(I(\tau))\right)^3}{I S f(I) (f(I))^2}\right) \leq \frac{f\left(I(\tau)\right)}{f\left(I^*\right)}+\frac{S f(I)}{S^* f\left(I^*\right)}+\frac{I S(\tau) \left(f(I(\tau))\right)^3}{I S f(I) (f(I))^2}-1 \leq 0,
$$
and
$$
\begin{aligned}
\frac{S V^*}{V S^*}+&\frac{V}{V^*}+\frac{S}{\left(\gamma_1+\mu\right) S^*}+\frac{h\left(I(\tau)\right)}{h\left(I^*\right)}+1+\frac{V h(I)}{V^* h\left(I^*\right)}+\ln \left(\frac{(V(\tau))^2\left(h\left(I(\tau)\right)\right)^2}{\left(\gamma_1+\mu\right) V^* h(I)}\right) \\
\leq & \frac{S V^*}{V S^*}+\frac{V}{V^*}+\frac{S}{\left(\gamma_1+\mu\right) S^*}+\frac{h\left(I(\tau)\right)}{h\left(I^*\right)}+\frac{V h(I)}{V^* h\left(I^*\right)}+\frac{V(\tau)^2\left(h\left(I(\tau)\right)\right)^2}{\left(\gamma_1+\mu\right) V^* h(I)}.
\end{aligned}
$$
Hence, $\frac{d H(t)}{d t} \leq 0$.
By applying LaSalle's invariance principle (See \cite{LaSalle1976}), we conclude that the endemic equilibrium point of system \eqref{E4.1} is globally asymptotically stable. The proof is complete.
\end{proof}


\section{Numerical results}\label{S6}

The numerical simulations for the proposed model \eqref{E2.1} with \eqref{E2.2} have been solved using MATLAB. We have employed the finite difference method (FDM) to approximate the solution of the PDE system in one spatial variable $x$ and time $t$. In the FDM, the spatial domain $0 \leq x \leq 1$ is discretized into a grid with $N + 1$ equally spaced points $x_k = kh$ for $k = 0, 1, . . . , N$, where $h = \frac{1}{N}$ denotes the uniform grid width. Moreover, the diffusion term in model \eqref{E2.1} has been discretized using a second-order central difference formula, as outlined by \cite{Mohd2019}
$$
\Delta y \approx \frac{y_{k+1} - 2y_k + y_{k-1}}{h^2}.
$$
This results in a system of $3(N + 1)$ ODEs, with one equation for each compartment discretized across uniformly spaced spatial nodes. The solution of this extensive ODE system is computed until a steady state is reached (i.e., until $t = 1500$). The grid spacing is set as $h = 10^{-2}$, and the parameter values are as indicated in Table \ref{Tab1} with $f(I) = \beta_1 I$ and $h(I) = \beta_2 I$.

\begin{table}[hbtp]
\centering
\caption{Parameter values and initial conditions of \eqref{E2.1}.}\label{Tab1}
\begin{tabular}{ccc}
\hline Symbol & Description & Value\\
\hline 
$\Lambda$ & Birth (or demographic) rate of the population & $0.392465$\\
$\beta_1$ & Transmission rate for the susceptible & $0.002$ (for $\mathcal{R}_0 > 1$)\\ 
& & $0.0008$ (for $\mathcal{R}_0 < 1$)\\
$\beta_2$ & Transmission rate for the vaccinated & $0.0016$ (for $\mathcal{R}_0 > 1$)\\ 
& & $0.00064$ (for $\mathcal{R}_0 < 1$)\\
$\alpha$ & The rate of vaccination of susceptible & $0.005$\\
$\mu$ & Natural death rates & $0.001$\\
$\gamma_1$ & Recovered rate of the vaccinated & $0.005$\\
$\gamma$ & Recovered rate of the infected & $0.009$\\
$c$ & Disease-induced death rate & $0.09$\\
$d_S = d_V = d_I = d_R$ & Diffusion coefficients & $0.1$\\ 
$S_0$ & Initial susceptible individuals & $30$\\ 
$V_0$ & Initial vaccinated individuals & $10$\\ 
$I_0$ & Initial infected individuals & $5$\\
$R_0$ & Initial recovered individuals & $0$\\
\hline
\end{tabular}
\end{table}

In Figure \ref{F1}, the condition $\mathcal{R}_0 = 0.8721 < 1$ is illustrated for the given parameters $\beta_1 = 0.0008$ and $\beta_2 = 0.00064$. In this scenario, the disease-free equilibrium is shown to be asymptotically stable, and the system's solutions converge towards a state with no disease.\\

On the other hand, in Figure \ref{F2}, the case where $\mathcal{R}_0 = 2.1804 > 1$ is depicted, considering the parameters $\beta_1 = 0.002$ and $\beta_2 = 0.0016$. In this context, the endemic equilibrium is demonstrated to be asymptotically stable, and the system's solutions tend towards a positive balance.

\begin{figure}[hbtp]
\centering
\includegraphics[scale=0.425]{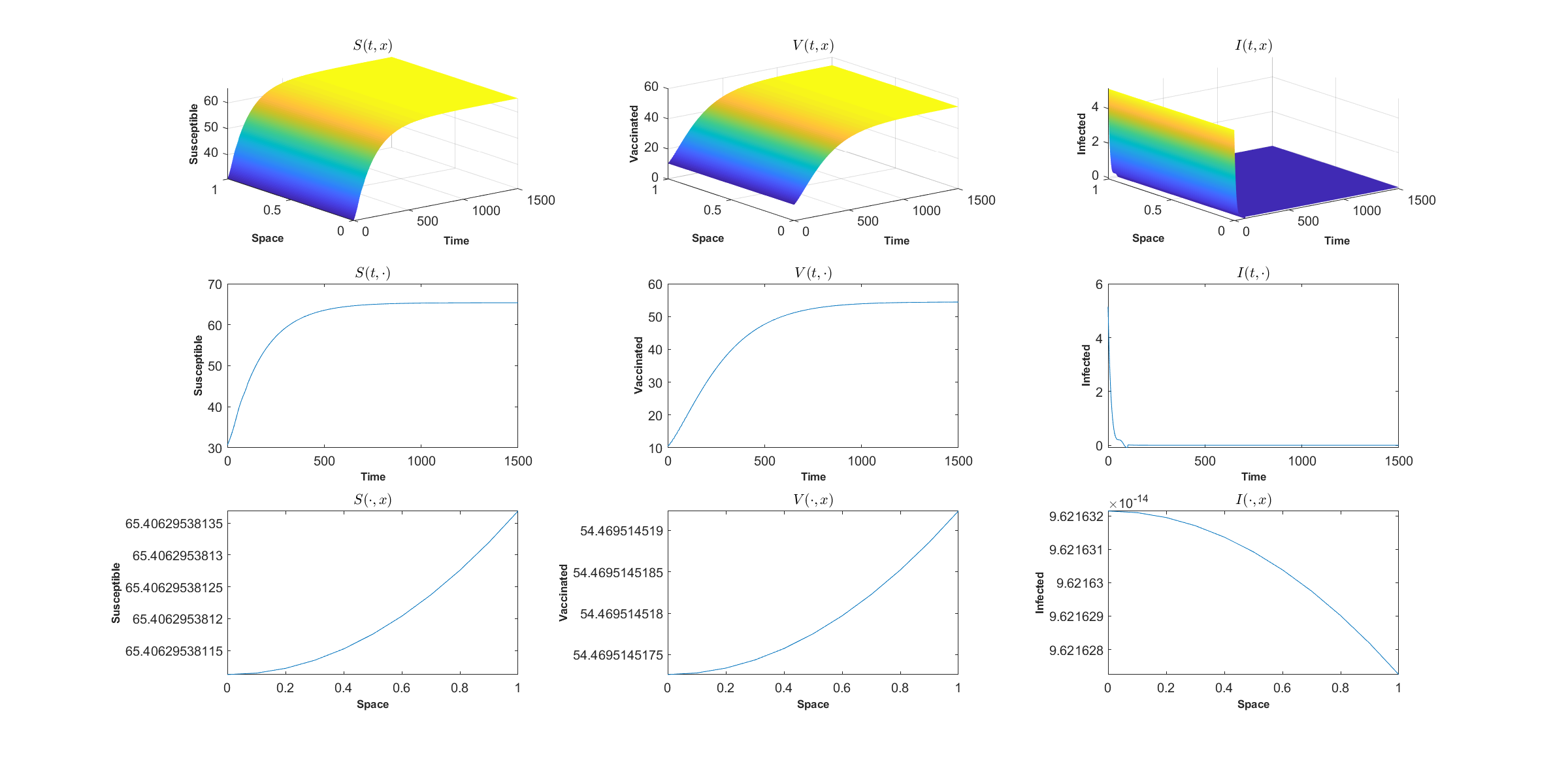}
\caption{Numerical results of \eqref{E2.1}-\eqref{E2.2} when $\mathcal{R}_0 < 1$.}\label{F1}
\end{figure}

\begin{figure}[hbtp]
\centering
\includegraphics[scale=0.425]{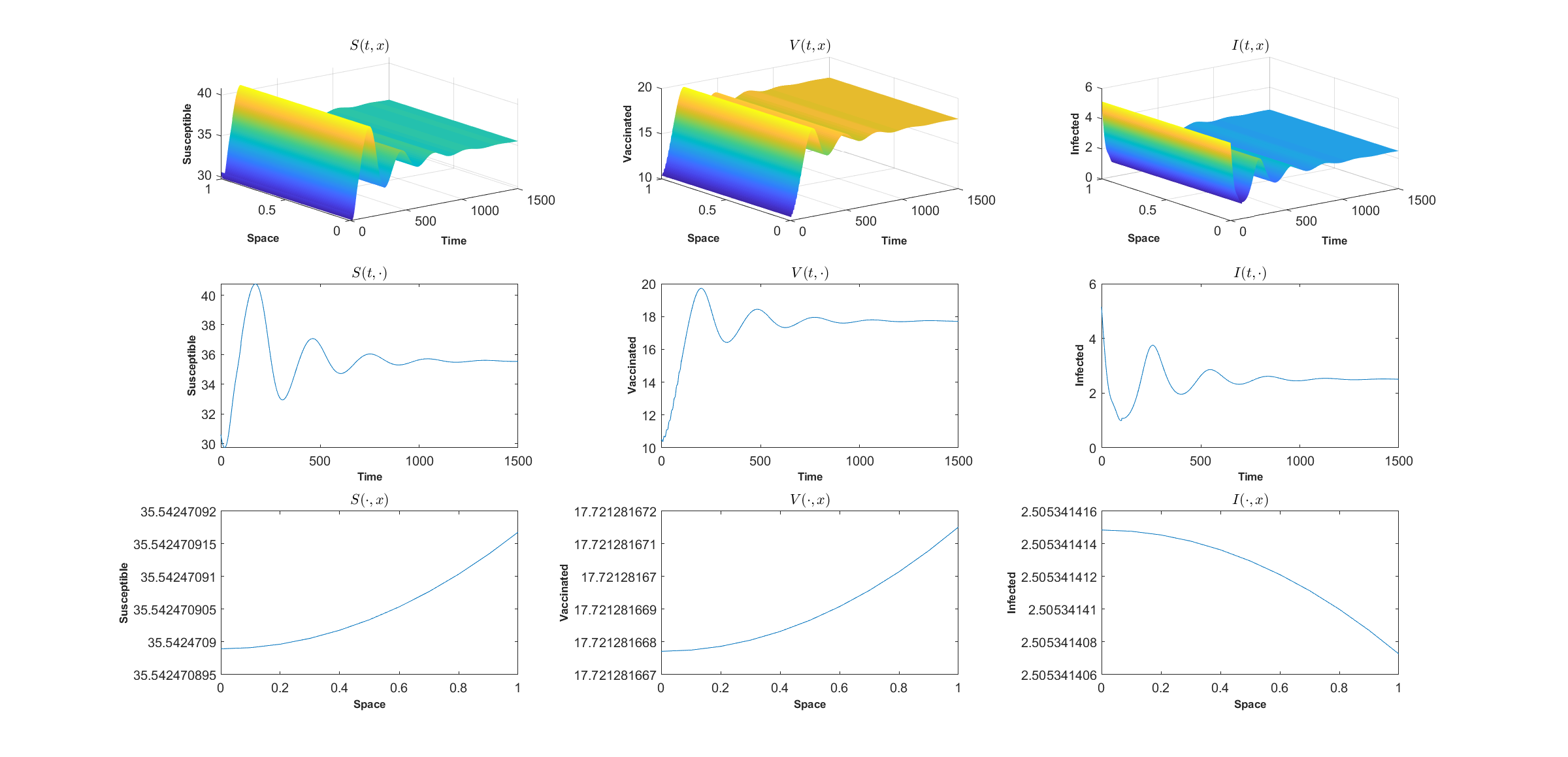}
\caption{Numerical results of \eqref{E2.1}-\eqref{E2.2} when $\mathcal{R}_0 > 1$.}\label{F2}
\end{figure}

\newpage

\section{Conclusion and discussion}\label{S7}

In conclusion, this paper has investigated a reaction-diffusion SVIR infection model with distributed delay and nonlinear incidence rate. The well-posedness of the model was established, and through Lyapunov functionals, we demonstrated the global asymptotic stability of both the disease-free and endemic equilibrium states. 
The incorporation of delay in the SVIR model with reaction-diffusion yields several significant insights into the dynamics of infectious diseases within spatially distributed populations. The introduction of delay mechanisms captures the time lag between infection, onset of symptoms, and subsequent interventions. Moreover, it leads to the emergence of new dynamical behaviors, such as the existence of multiple equilibrium states or the occurrence of sustained oscillations. These phenomena arise due to the interplay between the intrinsic disease dynamics, spatial diffusion, and the time delay in the system.
Numerical simulations further illustrated these theoretical findings, showcasing the behavior of the system under different scenarios of the basic reproduction number $\mathcal{R}_0$. Specifically, when $\mathcal{R}_0 < 1$, the disease-free equilibrium was shown to be asymptotically stable, leading to the eventual eradication of the disease. Conversely, for $\mathcal{R}_0 > 1$, the endemic equilibrium was found to be asymptotically stable, resulting in a persistent state of infection within the population. These results contribute to our understanding of the dynamics of infectious diseases and have implications for public health interventions aimed at disease control and prevention. Further research could explore the effects of additional factors, such as spatial heterogeneity or seasonal variations, on the dynamics of epidemic models with distributed delay, providing valuable insights for disease management strategies in real-world scenarios.


\section*{Declarations}

\subsection*{Acknowledgments}
The authors express their appreciation to the reviewers for their valuable comments.
This work is carried out under the supervision of CNRST as part of the PASS program.

\subsection*{Funding} 

The authors assert that no funding was received for the preparation and execution of this manuscript.

\subsection*{Data availability} 

All the information and data that were analyzed or generated to support the results of this work are provided within this manuscript.

\subsection*{Conflict of interest} 

The authors declare that there are no problems or conflicts of interest between them that may affect the study in this paper.


\bibliographystyle{acm}
\bibliography{paper}

\begin{thebibliography}{10}

\bibitem{Ahmed2021}
{\sc Ahmed, N., Elsonbaty, A., Raza, A., Rafiq, M., and Adel, W.}
\newblock {Numerical simulation and stability analysis of a novel
  reaction–diffusion COVID-19 model}.
\newblock {\em Nonlinear Dynamics 106}, 2 (June 2021), 1293–1310.

\bibitem{Anderson1979}
{\sc Anderson, R.~M., and May, R.~M.}
\newblock {Population biology of infectious diseases: Part I}.
\newblock {\em Nature 280}, 5721 (Aug. 1979), 361–367.

\bibitem{Annas2020}
{\sc Annas, S., Isbar~Pratama, M., Rifandi, M., Sanusi, W., and Side, S.}
\newblock {Stability analysis and numerical simulation of SEIR model for
  pandemic COVID-19 spread in Indonesia}.
\newblock {\em Chaos, Solitons \& Fractals 139\/} (Oct. 2020), 110072.

\bibitem{AvilaVales2021}
{\sc Avila-Vales, E., and P{é}rez, A. G.~C.}
\newblock {Dynamics of a reaction–diffusion SIRS model with general incidence
  rate in a heterogeneous environment}.
\newblock {\em Zeitschrift f\"{u}r angewandte Mathematik und Physik 73}, 1
  (Nov. 2021).

\bibitem{Chen2020}
{\sc Chen, Y.~C., Lu, P.~E., Chang, C.~S., and Liu, T.~H.}
\newblock {A Time-Dependent SIR Model for COVID-19 With Undetectable Infected
  Persons}.
\newblock {\em IEEE Transactions on Network Science and Engineering 7}, 4 (Oct.
  2020), 3279–3294.

\bibitem{Cui2007}
{\sc Cui, J., Sun, Y., and Zhu, H.}
\newblock The impact of media on the control of infectious diseases.
\newblock {\em Journal of Dynamics and Differential Equations 20}, 1 (May
  2007), 31–53.

\bibitem{Duan2014}
{\sc Duan, X., Yuan, S., and Li, X.}
\newblock {Global stability of an SVIR model with age of vaccination}.
\newblock {\em Applied Mathematics and Computation 226\/} (Jan. 2014),
  528–540.

\bibitem{Groeger2014}
{\sc Groeger, J.}
\newblock Divergence theorems and the supersphere.
\newblock {\em Journal of Geometry and Physics 77\/} (Mar. 2014), 13–29.

\bibitem{Guenther1996}
{\sc Guenther, R.~B., and Lee, J.~W.}
\newblock {\em Partial differential equations of mathematical physics and
  integral equations}.
\newblock Courier Corporation, 1996.

\bibitem{Jamiluddin2021}
{\sc Jamiluddin, M.~S., Mohd, M.~H., Ahmad, N.~A., and Musa, K.~I.}
\newblock {Situational Analysis for COVID-19: Estimating Transmission Dynamics
  in Malaysia using an SIR-Type Model with Neural Network Approach}.
\newblock {\em Sains Malaysiana 50}, 8 (Aug. 2021), 2469–2478.

\bibitem{Kumar2018}
{\sc Kumar, A., and Nilam}.
\newblock Stability of a time delayed {SIR} epidemic model along with nonlinear
  incidence rate and holling type-{II} treatment rate.
\newblock {\em International Journal of Computational Methods 15}, 06 (Sept.
  2018), 1850055.

\bibitem{LaSalle1976}
{\sc La~Salle, J.~P.}
\newblock {\em The stability of dynamical systems}.
\newblock Society for Industrial and Applied Mathematics, Jan. 1976.

\bibitem{Luo2019}
{\sc Luo, Y., Zhang, L., Zheng, T., and Teng, Z.}
\newblock Analysis of a diffusive virus infection model with humoral immunity,
  cell-to-cell transmission and nonlinear incidence.
\newblock {\em Physica A: Statistical Mechanics and its Applications 535\/}
  (Dec. 2019), 122415.

\bibitem{Marinov2022}
{\sc Marinov, T.~T., and Marinova, R.~S.}
\newblock {Adaptive SIR model with vaccination: simultaneous identification of
  rates and functions illustrated with COVID-19}.
\newblock {\em Scientific Reports 12}, 1 (Sept. 2022).

\bibitem{Martin1990}
{\sc Martin, R.~H., and Smith, H.~L.}
\newblock Abstract functional differential equations and reaction-diffusion
  systems.
\newblock {\em Transactions of the American Mathematical Society 321}, 1 (Sept.
  1990), 1.

\bibitem{Mohd2019}
{\sc Mohd, M.~H.}
\newblock {\em Numerical bifurcation and stability analyses of partial
  differential equations with applications to competitive system in ecology}.
\newblock Springer Singapore, 2019, p.~117–132.

\bibitem{Pazy1983}
{\sc Pazy, A.}
\newblock {\em Semigroups of linear operators and applications to partial
  differential equations}.
\newblock Springer New York, 1983.

\bibitem{Ren2018}
{\sc Ren, X., Tian, Y., Liu, L., and Liu, X.}
\newblock {A reaction–diffusion within-host HIV model with cell-to-cell
  transmission}.
\newblock {\em Journal of Mathematical Biology 76}, 7 (Jan. 2018), 1831–1872.

\bibitem{Salman2021}
{\sc Salman, A.~M., Ahmed, I., Mohd, M.~H., Jamiluddin, M.~S., and Dheyab,
  M.~A.}
\newblock {Scenario analysis of COVID-19 transmission dynamics in Malaysia with
  the possibility of reinfection and limited medical resources scenarios}.
\newblock {\em Computers in Biology and Medicine 133\/} (June 2021), 104372.

\bibitem{SidiAmmi2020}
{\sc Sidi~Ammi, M.~R., Tahiri, M., and Torres, D. F.~M.}
\newblock Global stability of a {C}aputo fractional {SIRS} model with general
  incidence rate.
\newblock {\em Mathematics in Computer Science 15}, 1 (Mar. 2020), 91--105.

\bibitem{SidiAmmi2023}
{\sc Sidi~Ammi, M.~R., Zinihi, A., Raezah, A.~A., and Sabbar, Y.}
\newblock Optimal control of a spatiotemporal $\mathcal{SIR}$ model with
  reaction{\textendash}diffusion involving $p$-laplacian operator.
\newblock {\em Results in Physics 52\/} (Sept. 2023), 106895.

\bibitem{Smith1995}
{\sc Smith, H.~L.}
\newblock {\em Monotone dynamical systems: An introduction to the theory of
  competitive and cooperative systems}.
\newblock American Mathematical Society, Mar. 1995.

\bibitem{Wang2021}
{\sc Wang, J., Zhang, R., and Kuniya, T.}
\newblock {A reaction–diffusion
  Susceptible–Vaccinated–Infected–Recovered model in a spatially
  heterogeneous environment with Dirichlet boundary condition}.
\newblock {\em Mathematics and Computers in Simulation 190\/} (Dec. 2021),
  848–865.

\bibitem{Yang2022}
{\sc Yang, B., Yu, Z., and Cai, Y.}
\newblock {The impact of vaccination on the spread of COVID-19: Studying by a
  mathematical model}.
\newblock {\em Physica A: Statistical Mechanics and its Applications 590\/}
  (Mar. 2022), 126717.

\bibitem{Zhang2019}
{\sc Zhang, C., Gao, J., Sun, H., and Wang, J.}
\newblock {Dynamics of a reaction–diffusion SVIR model in a spatial
  heterogeneous environment}.
\newblock {\em Physica A: Statistical Mechanics and its Applications 533\/}
  (Nov. 2019), 122049.

\bibitem{Zhou2007}
{\sc Zhou, Y., Xiao, D., and Li, Y.}
\newblock Bifurcations of an epidemic model with non-monotonic incidence rate
  of saturated mass action.
\newblock {\em Chaos, Solitons \& Fractals 32}, 5 (June 2007), 1903–1915.

\bibitem{Zinihi2024}
{\sc Zinihi, A., Sidi~Ammi, M.~R., and Ehrhardt, M.}
\newblock Optimal control of a diffusive epidemiological model involving the
  caputo-fabrizio fractional time-derivative, 2024.

\end{thebibliography}

\end{document}